\documentclass[11pt, twoside]{article}
\usepackage{amsmath,amsthm}
\usepackage{amsfonts}
\usepackage{amssymb,latexsym}
\usepackage{enumerate}
\newtheorem{theorem}{Theorem}[section]
\newtheorem{remark}[theorem]{Remark}
\newtheorem{proposition}[theorem]{Proposition}
\newtheorem{lemma}[theorem]{Lemma}
\newtheorem{definition}[theorem]{Definition}

\newtheorem{corollary}[theorem]{Corollary}

\def\n{\noindent}
\def\fr{\frac}
\def\Om{\Omega}
\def\pa{\partial}
\def\de{\delta}
\def\ve{\varepsilon}

\def\psh{plurisubharmonic}
\def\va{\phi}
\def\la{\lambda} 
\def\C{\mathbb C}
\def\c{\mathcal C}
\def\ov{\overline}
\def\R{\mathbb R}
\def\al{\alpha}
\begin{document}
	\setlength{\baselineskip}{18truept}
	\pagestyle{myheadings}
	\markboth{ N.Q. Dieu, T.V. Long and T.D. Hieu}{Dirichlet problem for plurisubharmonic functions}
	\title {Dirichlet problem for plurisubharmonic functions\\
		on bounded quasi B-regular domains in $\C^n$}
	\author{
		N.Q. Dieu, T.V. Long and T.D. Hieu
		\\Department of Mathematics, Hanoi National University of Education,\\ Hanoi, Vietnam;
		\\E-mail: ngquang.dieu@hnue.edu.vn\\tvlong@hnue.edu.vn\\tranduchieu1709@gmail.com
	}
	\date{}
	\maketitle
	
	\renewcommand{\thefootnote}{}
	
	
	
	\renewcommand{\thefootnote}{\arabic{footnote}}
	\setcounter{footnote}{0}
	
	{\it Dedicated to Professor Nguyen Van Khue on the occasion of his 85th birthday.}

	\begin{abstract}
		\n
		We investigate existence and uniqueness of maximal plurisubharmonic functions on bounded domains with boundary data that are not assumed to be continuous or bounded. The result is applied to approximate (possibly unbounded from above) plurisubharmonic functions by continuous quasi upper bounded ones. A key step in our approach is to explore
		continuity of the Perron-Bremermann envelope of plurisubharmonic functions that are dominated by a given function  $\phi$ defined on the closure of the domain.
	\end{abstract}
	\section{Introduction}
	Let $\Omega$ be a bounded domain in $\mathbb{C}^n $. By $SH(\Om), PSH(\Omega)$ we denote the cones of 
	subharmonic and plurisubharmonic functions on $\Omega,$ allowing functions to be identically equal to $-\infty.$
	Furthermore, $SH^* (\Om), PSH^*(\Om)$ denote the class of subharmonic and plurisubharmonic functions on $\Om$ that are not
	identically $-\infty.$
	A function $u\in PSH^* (\Omega)$ is said to be {\it maximal} if for every relatively compact open subset $G$ of  $\Omega$ and for every  $v\in PSH(G)$ such that
	$v^* \le u$ on $\partial G$ we have $v\le u$ on $G$. Here $v^*$ denotes the upper regularization of  $v$, i.e
	$$
	v ^*(z):=\limsup_{G \ni \xi \to z} v(\xi), \ z\in \overline G.
	$$
	Similarly, we also define the lower regularization of  $v$ by
	$$v_*(z):=\liminf_{G \ni \xi \to z} v (\xi), \ z\in \overline G.$$
	We use the symbol $MPSH(\Omega)$ to denote the family of all maximal plurisubharmonic functions on  $\Omega$.
	Now let $\phi: \pa \Om \to \mathbb R$ be a boundary data. The classical Dirichlet problem concerns with  the existence and uniqueness of the problem 
	$$\begin{cases}
		u \in MPSH(\Om)\\
		\lim\limits_{z \to x} u(z)=\phi (x) \ \forall x \in \pa \Om.
	\end{cases}$$
	In case, $\Om$ is strictly pseudoconvex, the existence and uniqueness of solution to the above Dirichlet problem has been confirmed in \cite{Br59}. Moreover, in this case, Bedford and Taylor proved in \cite{BT}, that the solution $u$
	also satisfies the homogeneous Monge-Amp\`ere equation $(dd^c u)^n=0$ (understood in a weak sense).
	Later on, the Dirichlet problem was studied throughly in \cite{Si} and then \cite{Blocki}
	when $\Om$ is a general bounded domain in $\C^n.$ 
	This line of research culminates in \cite{Si} with the following basic notion of \(B\)-regularity.
	A domain \(\Omega\subset\mathbb{C}^n\) is called \(B\)-regular if every real-valued continuous
	function \(\phi\) on \(\partial\Omega\) extends to a function \(u\in \mathrm{PSH}(\Omega)\cap C(\overline{\Omega})\)
	such that \(u|_{\partial\Omega}=\va\).
	In particular, \(B\)-regular domains are precisely the domains on which the Dirichlet
	problem for plurisubharmonic functions with continuous boundary data admits a unique (continuous) solution.
A sufficient condition for a domain $\Omega\subset\C^n$ to be $B$-regular is that its boundary $\partial\Omega$ contains no complex variety of positive dimension.
	The reader may consult \cite{DDH} for specific examples of bounded $B-$regular domains.
	When the boundary datum \(\va\) is allowed to take the values \(\pm\infty\), the Dirichlet
	problem becomes subtler. For instance, let \(\Omega=\mathbb D\subset\mathbb C\) be the unit disk and
	let \(\va:\partial\mathbb D\to[-\infty, 0)\) be a continuous function such that
	\(\va \equiv -\infty\) on a set \(E\subset\partial\mathbb D\) of positive measure.
	Then, by the mean–value inequality together with a routine limiting argument, we deduce the non-existence of a harmonic function \(u\) on \(\mathbb D\) with
	\[
	\lim_{z\to x} u(z)=\va (x)\qquad\text{for all }x\in\partial\mathbb D.
	\]
	Thus, it is natural to relax the boundary condition slightly. More precisely, let 
	\(\va: \partial\Omega\to[-\infty,\infty]\) be a boundary datum and consider the following
	generalized Dirichlet problem:
	\[
	\begin{cases}
		u \in \mathrm{MPSH}(\Omega),\\[2mm]
		\displaystyle\lim_{z\to x} u(z)=\va (x)\quad \text{for all } x\in \partial\Omega\setminus E,
	\end{cases}
	\]
	where the exceptional set \(E\subset \partial\Omega\) is \(\Omega\)-pluripolar (see the next section for the precise definition).
	Analogous to the classical Perron envelope in potential theory, the Perron–Bremermann
	envelope is a fundamental tool for solving the Dirichlet problem in pluripotential
	theory. In detail, given boundary data \(\varphi\) on \(\partial\Omega\), we consider the
	Perron–Bremermann envelope
	\[
	P_\va (z)
	:= \sup\bigl\{ u(z) : u \in \mathrm{PSH}(\Omega),
	\ \ u^*|_{\partial\Omega}\le \va \bigr\}, \qquad z\in\Omega.
	\]
	We then study \(P_\va \) by establishing its regularity and its boundary
	behavior. For technical convenience, we work with envelopes of functions defined on
	the closure \(\overline{\Omega}\).
	Given a function $\phi:\overline{\Omega}\to [-\infty,+\infty]$, the Perron-Bremermann envelope of  $\phi$ is defined by
	$$
	P_{\phi}(z) =\sup\{u(z): u\in PSH(\Omega): u^*\le \phi\ \text{on}\ \overline{\Omega}\}.
	$$
	In \cite{Br59} Bremermann proved that $P_\phi$ assumes the boundary value $\phi$ continuously when $\Omega$ is strictly pseudoconvex and $\phi$ is real valued, continuous on $\partial \Omega$.
	It was later shown by Walsh \cite{Wal} that  $P_\phi\in C(\overline{\Omega})$. Thus $P_\phi$ is the unique maximal continuous plurisubharmonic function on $\Om$ that assumes the boundary values $\phi$ on $\Om.$ 
	Sufficient conditions ensuring H\"older continuity of $P_{\phi}$ was also studied recently in \cite{LHL}.
	In the case $\phi$ is discontinuous or  unbounded, the Dirichlet problem becomes more complicated.
	Moreover, in \cite{NiWi}, M. Nilsson and F. Wikstr\"om have shown  that if  $\Omega$ is $B$-regular and $\phi$ is  tame plurisuperharmonic on $\Omega$ such that $\phi^*=\phi_*$ on $\overline{\Omega}$ then $P_\phi$ is the unique maximal plurisubharmonic function
	that is continuous outside a pluripolar set
	with the correct  boundary value $\phi$.
	Moreover, if  $v$ is a non-trivial strong plurisuperharmonic majorant  to $\phi$, M. Nilsson \cite{Ni} showed that  $P_\phi$ is continuous outside the singularities  of  $v$.
	Note that in order to prove the continuity of  $P_\phi,$ the authors in \cite{NiWi} and \cite{Ni}, used  variations of Edwards' duality theorem and the approximation property of $B$-regular domains.
	
	Given the relaxed boundary condition on \(u\), it is reasonable to replace \(B\)-regularity by a slightly weaker boundary condition.
	\begin{definition}
		A bounded regular domain $\Om$ in $\C^n$ is said to be quasi $B-$regular 
		for every continuous function
		$\phi: \pa \Om \to \R$ there exist a $\Om-$pluripolar set $A \subset \pa \Om$ and a quasi upper bounded function
		$u \in PSH(\Om)$ such that 
		$$\lim\limits_{z \to x} u(z)=\phi (x), \ \forall x \in (\pa \Om) \setminus A.$$
	\end{definition}
	\begin{remark} {\rm (i) Let $\Om$ be the worm domain constructed by Diederich and Fornaess (see Proposition 1 in \cite{DF}). Then $\Om$ is a $\c^\infty-$ smooth, bounded pseudoconvex in $\C^2$. Moreover, $\pa \Om$ fails to be strictly pseudoconvex  
			exactly on the disk $\{1 \le |z| \le r\} \times \{0\}.$ It follows that $\Om$ is pseudoconvex, quasi B-regular but not B-regular.
			
			\n 	
			(ii) A priori, the exceptional set $A$ on the definition of a quasi $B-$regular domains depends on $\phi,$ we will show, however, in Theorem \ref{breg} below that $A$ might be chosen to be independent of $\phi.$  We also study continuity of $P_\phi$ on such domains when $\phi$ is real valued continuous on $\ov \Om.$
			}
	\end{remark}
	\begin{theorem} \label{breg}
	Let $\Om$ be a bounded quasi $B-$regular domain in $\C^n.$ 
	Then the following assertions hold true:
	
	\n 
	(a) There exists a $\Om-$pluripolar set $A \subset \pa \Om$ such that for every continuous function
	$\phi: \ov \Om \to \R$ and every $x \in (\pa \Om) \setminus A,$ we have
	\begin{equation} \label{eq000}
		\lim\limits_{z \to x} \tilde P_\phi (z)= \lim\limits_{z \to x} P_\phi (z)=\phi (x),
	\end{equation}
	where
	\begin{equation} \label{eq00}
		\tilde P_{\phi} (z) := \sup \{u (z): u \in PSH(\Om), u^*|_{\pa \Om} \le \phi\}, \ z \in \Om. 
	\end{equation}
	\n 
	(b) There exists a $\Om-$pluripolar set $X \subset \ov \Om$
	such that for every continuous function $\phi: \ov \Om \to \R,$ we have:
	
	\n 
	(i) $P_\phi \in PSH(\Om) \cap L^\infty (\Om);$ 
	
	\n 
	(ii) $P_\phi (z) \to \phi (x)$ as $z \to x$ for all $x \in (\pa \Om) \setminus X$.
	
	\n 
	(iii) $P_\phi$ is continuous at every point in $\Om \setminus X.$ 
\end{theorem}	
\n 
Motivated from Theorem \ref{breg} we have the following
\begin{definition}
Let $\Om\subset\C^n$ be a bounded, quasi $B$-regular domain. The $B$-singular set of $\Om$, denoted $B_\Om$, is the intersection of all $\Om$-pluripolar sets $A$ satisfying property~(a) of Theorem~\ref{breg}.	
\end{definition} 
\n
It will follow from the {\it proof} of Theorem \ref{breg} that the exceptional set $X$ might be chosen to be $\widehat {B_\Om},$ the $\Om-$pluripolar hull of $B_\Om.$
Theorem \ref{breg}, is then used to study continuity of Perron-Bremermann envelopes on quasi $B-$regular domains, and in a slightly more general context
	where the initial function $\phi: \Om \to [-\infty, \infty]$ is {\it nearly continuous}, that is to say, when the exceptional set 
	$$
	E_\phi =\{z\in\overline{\Omega}:  \phi^*(z) \ne \phi_*(z)\}
	$$
	is  $\Omega$-pluripolar.
	\begin{theorem}\label{main2}
		Let $\Omega$ be a bounded  quasi B-regular domain in $\mathbb{C}^n$ and  $\phi: \Om \to [-\infty, \infty]$ be a nearly continuous function.
		Suppose that $\phi^*$ admits a superharmonic majorant $v \not \equiv+\infty$ on $\ov \Om.$
		Then the following assertions hold true:
		
		\n 
		(i) The  envelope   $$
		P_{\phi^*}(z) =\sup\{u(z): u\in PSH(\Omega): u^*\le \phi^*\ \text{on}\ \overline{\Omega}\}
		$$
		is plurisubharmonic on $\Om$.
		
	\n 
	(ii)  If  $\phi$ is bounded from below  by
	$\theta \in PSH^* (\Om)$
	and if $v$ is  a strong plurisuperharmonic majorant of $\phi^*$
	then $P_{\phi^*} \in PSH^* (\Om)$ and is quasi upper bounded on $\Omega$.
	
	\n 
	(iii) If $\theta$ is bounded from below on $\Om$ then 
	$$
	\lim\limits_{\Omega\ni  z\to  z_0}P_{\phi^*}(z)=\phi(z_0),\ \forall z_0\in \partial \Omega\setminus (B_\Om \cup E_\phi).
	$$
	Moreover, $P_{\phi^*}$ is  continuous at every point in the set $\Omega \setminus Y$, where
	$$Y=\{z \in \ov \Om: v^*(z)=+\infty \}\cup \widehat {B_\Om} \cup \widehat {E_\phi}.$$
	\end{theorem}
	\begin{remark} {\rm 
			(a) An appealing direction in (iii) is to drop the lower boundedness assumption on \(\theta\). However, our techniques currently do not accommodate this extension.
			
			\n 
			(b) We wish to emphasize that the second statement in (iii) is a bit stronger than continuity of \(P_{\phi^*}\) on \(\Omega\setminus Y\).
			
			\n 
			(c) For \(B\)-regular domains with \(E_\phi=\emptyset\), Theorem~\ref{main2} was essentially proved by M. Nilsson in \cite{Ni}.  Some key ingredients in his approach are a variant of Edwards' theorem, along with the approximation property specific to 
			$B$-regular domains (in \cite{Wi}). In contrast, our method relies heavily on Theorem \ref{breg} and
			the insertion lemma, which enables us to insert a continuous function between a lower semicontinuous function and an upper semicontinuous one.
			}
	\end{remark}
	
	\n
	Using Theorem \ref{main2} we are able to establish the following approximation result for unnecessarily upper bounded plurisubharmonic functions.
	\begin{theorem} \label{main3}
		Let $\Om$ be a bounded quasi $B-$regular domain and $u$ be a function in $PSH^* (\Om).$ Assume that 
		$((\max \{u, 0\})^\al)^*$ admits a nearly continuous, plurisuperharmonic majorant $\psi \ge 0$ on $\ov \Om$ for some constant $\alpha>1.$
		Suppose also that $Y \cap \Om$ is contained in a closed pluripolar subset $Z$ of $\Om$ where
		$$Y:= \{z \in \ov \Om: \psi^*=+\infty\} \cup \widehat{B_\Om} \cup  \widehat {E_\psi}.$$   
\n
Then there exist a sequence $u_j \in PSH^* (\Om)$
		having the following properties:
		
		\n 
		(i) $u_j$ is quasi upper bounded
		on $\Om;$ 
		
		\n 
		(ii) $u_j$ decreases to $u$ on $\Om;$
		
		\n 
		(iii) $u_j$  is real valued and continuous at every point in $\Om \setminus Z.$ 
	\end{theorem}
	\begin{remark} {\rm (a) If one does not require the \(u_j\) to be quasi upper bounded, the conclusion is an immediate consequence of the Forn{\ae}ss–Narasimhan global approximation theorem (\cite{FN}) in the case where \(\Omega\) is pseudoconvex. 
			
\n 
(b) If $u$ is bounded from above on $\Om$ then the existence of $\psi$ is obvious. In this case,
			Theorem~\ref{main2} follows from a combination of Theorem 4.3 and Theorem 3.1 in \cite{DW}, which in turn rests on Edwards’ duality theorem, used crucially in \cite{Wi}.
			
\n 
(c) If $\Omega$ is bounded and $B$-regular, then Theorem 4.1 in \cite{Wi} provides an even stronger result: every \emph{negative} plurisubharmonic function $u$ on $\Omega$ can be approximated from above on $\overline{\Omega}$ by a decreasing sequence ${u_j} \subset PSH(\Omega)\cap C(\overline{\Omega})$, i.e., $u_j^* \searrow u^*$ on $\overline{\Omega}$.
In this case, the method of \cite{Wi} is somewhat more direct than ours, in that the $u_j$ are obtained by
a delicate patching of plurisubharmonic functions constructed from $u$. 
			
\n 
(d) Motivated by the above theorem of Wikstr\"om, a natural question is to what extent a quasi upper bounded plurisubharmonic function on \(\Omega\) can be approximated from above by continuous quasi upper bounded ones.
At present, however, we are not aware of any method to tackle this problem.
		}
	\end{remark}	
	\n
	We end up with the following result on solvability and uniqueness of the
	Dirichlet problem on bounded quasi $B-$regular domains with (possibly) discontinuous boundary values.
	\begin{theorem}\label{main4}
		Let $\Omega$ and  $\phi$ as in Theorem \ref{main2} (iii). Assume, in addition, that $\phi$ is superharmonic on $\Om.$
		Then
		$P_{\phi*}$ is the unique quasi upper bounded solution of the Dirichlet problem
		\begin{equation}\label{Dirichlet}
			\begin{cases}
				u \in MPSH(\Omega);\\
				\inf\limits_{\Om} u>-\infty;\\
				\lim\limits_{\Omega\ni  z\to x}u(z)=\phi(x),\ \forall x \in \pa \Om \ \text{except for a}\ \Om-\text{pluripolar set}
			\end{cases}
		\end{equation}
	\end{theorem}
	\n
	We should say that, using Edwards' duality theorem, a similar result was proved in Theorem 1.6 of \cite{DLS} in the case where $\va$ is real-valued and continuous on $\partial \Omega$.
	
	\n
	The paper is organized as follows. In the next section we gather up some notations and auxiliary facts about $\Omega$-pluripolar sets,  functions  having   non-trivial strong majorant and quasi bounded plurisubharmonic functions.
	The most important one is a sort of the comparison principle (Lemma \ref{extended maximal}).
	Section 3 is devoted to proving the main results and presenting examples to which our results apply.
	
	\n
	We conclude this introduction with the following open problem: investigate an analogue of
	Theorem~\ref{main2} in the case where \(\phi\) admits only a (not necessarily strong)
	plurisuperharmonic majorant on \(\Omega\).
	This problem is closely connected to the continuity of the Green function \(G_A\) whose
	poles lie on a complex hypersurface \(A\subset\Omega\) defined by a holomorphic function \(f\).
	Indeed, under mild assumptions on \(f\), one may represent \(G_A\) as \(\log|f|\) plus the
	solution of the Dirichlet problem with boundary datum \(-\log|f|\).
	This question was first raised in \cite{LS} and subsequently studied in certain special
	cases in \cite{Ng}. To date, the continuity of \(G_A\) on \(B\)-regular domains remains
	largely open.
	\section{Preliminaries}
	We recall the following notion which was stated in \cite{Ni} and used implicitly in \cite{DLS}.
	\begin{definition} \label{polar}
		Let $A$ be a subset of $\overline{\Omega}$. Then we say that:
		
		\n 
		(i) $A$ is  $\Omega$-pluripolar  if
		there exists  $w \in PSH^*(\Om), w<0$ such that $ w^*|_{A}=-\infty;$
		
		\n 
		(ii) $A$ is  $\Omega$-polar  if
		there exists  $w \in SH^*(\Om), w<0$ such that $ w^*|_{A}=-\infty;$	
	\end{definition}
	\n
	If $A\subset \partial\Omega$ then  the above notion reduces to that  of b-pluripolarity formulated in \cite{DW}.
	Obviously every pluripolar set in $\Omega$ is $\Omega$-pluripolar.
	It is also known \cite{DW} that  a set  $A\subset \overline{\mathbb{D}}$ where $\mathbb{D}$ is the unit disc in  $\mathbb{C}$ is $\mathbb{D}$-polar if and only if $A\cap \mathbb{D}$ is polar and  $ A\cap \partial \mathbb{D}$
	has arc length $0$.\\
	It is also important to see how $\Om-$pluripolar sets may propagate.
	\begin{definition} \label{hull}
		Let $A \subset \ov{\Om}$ be a $\Om-$pluripolar set. Then the $\Om-$pluripolar hull of $A$ is defined as	
		$$\hat {A}=\bigcap \{z\in\overline{\Omega}: w^*(z)=-\infty, w \in PSH^* (\Omega), w<0, w^*|_{A} \equiv -\infty \}.$$
	\end{definition}
	\n
	The following lemmas collect some simple but useful properties of $\Om-$pluripolar and $\Om-$polar sets.
	\begin{lemma} \label{union}
		Suppose that the sets \( \{A_j\} \) are \( \Omega \)-pluripolar. Then their union \( \bigcup A_j \) is \( \Omega \)-pluripolar as well.
	\end{lemma}
	\begin{proof}
		We include a short proof only for the shake of completeness. By Definition \ref{polar}, for each $j$, we can find $u_j \in PSH^* (\Om), u_j<0$ such that 
		$${u_j}^* (x)=-\infty, \ \forall x \in A_j.$$
		Since $u_j^{-1} (-\infty)$ has Lebesgue measure $0$ for every $j,$ we can find $z_0 \in \Om$ such that
		$u_j (z_0)>-\infty$ for all $j.$
		Set $$u(z):=\sum\limits_{j=1}^\infty -\fr1{2^j u_j (z_0)}u_j (z).$$ 
		Since $u(z_0)>-1$ we infer that $u \in PSH^*(\Om)$. Finally, for every $x \in A_j$ we have
		$$u^* (x) \le -\fr1{2^j u_j (z_0)} {u_j}^* (x)= -\infty.$$
		Thus we are done.
	\end{proof}
	\n
	The following notion, which originates from its potential-theoretic counterpart in \cite{Pa}, is coined "quasi bounded quasi everywhere"  in \cite{NiWi}. We change the terminology a bit only for the sake of brevity. 
	\begin{definition}
		A  function $u \in PSH(\Omega)$ is called quasi upper bounded  if
		there exists a sequence $\{u_j \}$ of upper bounded,  plurisubharmonic functions on $\Om$ such that $u_j \nearrow u$ quasi everywhere (i.e., except for a pluripolar set ) on $\Omega$ .
	\end{definition}
	\n
	Similarly, $u \in SH(\Om)$ is said to be quasi upper bounded if $u$ can be approximated monotonically from below
	by a sequence of upper bounded subharmonic functions outside a polar subset of $\Om.$
	
	\n 
	\begin{remark} 
		{\rm (a) Obviously every upper bounded plurisubharmonic function is  quasi upper bounded. On the other hand,
			by Theorem 3.2 in \cite{NiWi},  $\phi(z,w)=(-\log|z|)^\alpha, 0<\alpha<1$ is quasi upper bounded plurisubharmonic in the unit ball  $\mathbb{B}\subset \mathbb{C}^2$ but not bounded.
			
			\n 
			(b)
			More interestingly, the Poisson kernel on the  unit disk $\mathbb D$ is unbounded harmonic
			which is not quasi upper bounded (see Example 1.2 in \cite{NiWi}).
			
			\n 
			(c) A function $u \in PSH(\Om)$ is quasi upper bounded if and only if the sequence
			$$v_j:=\sup \{v(z): v \in PSH (\Om), v \le \min \{u,j\} \}$$ increases to $u$ quasi everywhere on $\Om.$
	Indeed, first observe that $v_j=v_j^* \in PSH(\Om),$ since $\min \{u,j\}$ is upper semicontinuous on $\Om$. 
	Thus, if $v_j \nearrow u$ quasi everywhere on $\Om$ then $u$ is quasi upper bounded. Conversely, if $u$ is quasi upper bounded, then we can find a sequence $u_j$ of upper bounded,  plurisubharmonic functions on $\Om$ such that $u_j \nearrow u$ quasi everywhere. Then, for every $j$ we can find $k(j)$ so large such that $u_j \le \min \{u, k(j)\}$ on $\Om.$ It follows that $u_j \le v_{k(j)} \le u$ on $\Om.$ Hence $v_{k(j)} \nearrow u$ quasi everywhere on $\Om,$
	and so is the whole increasing sequence $\{v_j\}$. 
This argument is implicit in Theorem~4.8 of \cite{NiWi}; we record it here for the reader’s convenience.
}	
	\end{remark}
	\n
	Now we are able to formulate the following version of the comparison principle for quasi upper bounded functions.
	\begin{lemma}\label{extended maximal}
		(i) Let $u \in PSH(\Om)$ be a quasi upper bounded function. Assume that $\va \in MPSH(\Om)$ is a function that satisfies
		$\va_*>-\infty$ on $\pa \Om$ and that
		$$u^*(z_0)\le \va_* (z_0),\ \forall  z_0\in \partial \Omega \setminus A,
		$$
		where $A$ is a $\Omega$-pluripolar  set. Then $u\le \va$ on $\Omega$.
		
		\n 
		(ii) Let $u \in SH(\Om)$ be a quasi upper bounded function. Assume that $\va$ is a harmonic function on $\Om$ that satisfies
		$\va_*>-\infty$ on $\pa \Om$ and that
		$$u^*(z_0)\le \va_* (z_0),\ \forall  z_0\in \partial \Omega \setminus A,
		$$
		where $A$ is a $\Omega$-polar  set. Then $u\le \va$ on $\Omega$.
	\end{lemma}
	\begin{proof} We only prove (i) since the proof of (ii) is similar.
		First we assume that $u$ is (truly) upper bounded on $\Om.$ Since $A$ is $\Omega$-pluripolar,
		there exists $v \in PSH^* (\Omega), v<0$ such that $ v^{*}|_A=-\infty$.
		Suppose that there exists $z_0 \in \Om$ with $u(z_0)> \va (z_0).$ 
		Since $u, \va$ are plurisubharmonic, we can assume that $v(z_0)>-\infty.$
		Fix $\varepsilon >0,$ we set 
		$$u_{\varepsilon}(z)=u (z)+\varepsilon v(z), \; \; z \in \Omega.$$
		We claim that there exists a relatively compact domain $G \subset \Om$ such that $z_0 \in G$ and
		\begin{equation} \label{eq01}
			u_{\varepsilon} \le \va+ {\de} \ \text{on}\ \pa G,
		\end{equation}
		where $\de:= \fr1{2}(u(z_0)-\va(z_0)).$
		If this is not so, then there exists a sequence $z_j \to \pa \Om$ such that
		$$u_{\varepsilon} (z_j) > \va (z_j)+\de.$$
		By compactness, we may	achieve that $z_j \to x_0 \in \pa \Om$. By considering two cases $x_0 \in A$ and 
		$x_0 \in (\pa \Om) \setminus A$ while using the assumption, we obtain contradictions.
		Thus, we may find  a relatively compact domain $G \subset \Om$ with $z_0 \in G$ and satisfies (\ref{eq01}).
		By maximality of $\va$ we deduce that $u_{\varepsilon} \le \va+\de \ \text{on}\  G.$
		So
		$$u(z_0)+\ve v(z_0) \le \va (z_0)+\de.$$
		By letting $\ve \to 0$ we reach a contradiction.
		Thus the lemma is settled in the case $u$ is bounded from above.
		Finally, we deal with the case where $u$ is quasi upper bounded. Let $u_j \subset PSH(\Om)$ be a sequence of upper bounded functions that increases to $u$ except for a pluripolar set $A'.$ Then, since $u_j^* \le \va_*$ on $(\pa \Om) \setminus A$, by the forgoing proof we conclude that $u_j \le \va$ on $\Om.$ By letting $j \to \infty$ we see that $u \le \va$ on $\Om \setminus A'.$ Since $A'$ is pluripolar, this inequality continues to hold on $\Om$ entirely. The desired conclusion follows.
	\end{proof}
\begin{remark} {\rm The hypothesis that \(u\) is quasi upper bounded is essential: 
		if \(u\) is chosen as the Poisson kernel on the unit disk \(\mathbb D\) and \(\phi\equiv 0\), the lemma does not hold.}
\end{remark}	
	\n
	Following \cite{NiWi} (see also \cite{Ni}), we use the next concept, which is central to our work.
	\begin{definition}\label{majorant}{\rm
			Let $f : \ov \Omega \to [-\infty,+\infty]$ be a function. 
			A lower semicontinuous function $\psi: \ov \Om \to (-\infty, \infty]$ is called a {\em strong plurisuperharmonic (resp. superharmonic) majorant} to $f$ if
			$-\psi \in PSH^*(\Om)$ (resp. $-\psi \in SH^* (\Om)),$ 
			$$\{f=+\infty\}\subset \{\psi=+\infty\}$$
			and
			$$
			\frac{\psi (z)}{f(z) } \to +\infty \ \text{ as}\ f (z) \ \to +\infty.
			$$
	}\end{definition}
	\n
	If $f$ is bounded from above on $\overline\Omega$ then obviously every plurisuperharmonic function $\psi$ in $\Omega$ is a strong  majorant to $f$.
The significance of strong plurisubharmonic majorants is made precise in the next lemma.
	\begin{lemma}\label{majorant lemma}
		Let $u \in PSH (\Om)$ be a function such that $u^*$ admits
	a lower semicontinuous function	$\psi: \ov \Om \to (-\infty, \infty]$ as a strong plurisuperharmonic majorant. Then the following assertions hold true:
		
\n 
(i)	For every $\delta>0$,the function $u-\de \psi$ is upper bounded plurisubharmonic on $\Om;$

\n 
(ii) $u$ is quasi upper bounded on $\Omega$.
\end{lemma}
\begin{proof}
The following simple proof is included for the convenience of the reader.

\n 
(i) Set $a:=\inf\limits_{z\in \overline{\Omega}}\delta\psi(z)>-\infty$.
Since
		$$
		\lim\limits_{u(z)\to+\infty}\frac{\delta \psi(z)}{u(z)}=+\infty
		$$
		we can find  $N>\max\{a,0\}$ such that
		$$
		\frac{\delta\psi(z)}{u(z)}>1 \ \text{as}\ u(z)>N +a.
		$$
		It follows that  $$u(z)\le \delta \psi(z)+N,\ \forall z\in \Omega.$$
	Thus	 $u-\delta\psi \in PSH^* (\Om)$ and
		$
		u(z)-\delta \psi(z)\le N
		$
		for all $z$ outside the pluripolar set $\{z\in\Omega:  \psi(z)=+\infty\}$. Thus $u(z)-\delta \psi(z)\le N$ on the whole of $\Omega$.
		
\n 
(ii) For $j \ge 1$ we set
		$$
		u_j:=f-\frac1j(\psi-\inf\limits_{\overline{\Omega}} \psi).
		$$
By applying (i) to $\de=1/j,$ we see that $u_j$ is upper bounded plurisubharmonic and $u_j\nearrow f$ quasi everywhere on $\Omega$.
	\end{proof}
\n
 Our final auxiliary result, known as the insertion Lemma, is due to Kat{\v{e}}tov-Tong (see \cite{Tong}).
 This lemma will be used in the proof of Theorem \ref{main2}.
	\begin{lemma} \label{insertion} 
		Let $X$ be a metric space and $u: X\to [-\infty, \infty)$ and $v: X \to (-\infty, \infty]$ be 
		upper semicontinuous and lower semicontinuous functions on $X,$ respectively.
		Assume that
		\[
		u(x) \le v(x)\qquad \text{for all }x\in X.
		\]
		Then there exists a real valued continuous function $h:X\to\mathbb{R}$ satisfying
		\[
		u(x) \le h(x) \le v(x)\qquad \text{for all }x\in X.
		\]
	\end{lemma}
\n	
We complete this section by presenting a sufficient condition that guarantees the $B$-regularity of a quasi $B$-regular domain. Before proceeding, let us recall the following fundamental notion, originally introduced by Sibony in \cite{Si}.
\begin{definition}
A compact set $K \subset \C^n$ is said to be $B-$regular if every real valued continuous functions on $K$ can be approximated uniformly on $K$ by continuous plurisubharmonic functions on neighborhoods of $K$.	
\end{definition}
\begin{proposition}
Let $\Om \subset \C^n$ be a bounded pseudoconvex domain with $\c^1$ smooth boundary. Suppose that $\Om$ is quasi $B-$ regular and the singular set $B_\Om$ is contained in a compact $B-$regular set $X \subset \pa \Om.$
Then $B_\Om=\emptyset,$ i.e., $\Om$ is $B-$regular.
\end{proposition}
\begin{proof} 
Fix $z_0 \in (\pa \Om) \setminus X,$ we claim that $\pa \Om$ is {\it locally} $B-$regular at $z_0.$
Since $\Om$ is $\c^1$ smooth, we can find a small $r_0>0$ such that $(\ov {\mathbb B} (z_0, r_0) \cap \pa \Om) -t\bold n_{z_0}$ is contained in $\Om,$ for all 
$t \in (0,r)$ where $\bold n_{z_0}$ is the outward normal vector at $z_0$ to $\pa \Om$. By shrinking $r_0$ further
we may assume $\ov {\mathbb B} (z_0, r_0) \cap X=\emptyset.$
Now, by the definition of $B_\Om,$ for every continuous function $\phi$ on the compact set $K:=\ov {\mathbb B} (z_0, r_0) \cap \pa \Om$ 
we can find $u \in PSH(\Om)$ such that 
$$\lim_{z \to x} u(z)=\va (x) \ \forall x \in K.$$
So for every $\ve>0$, there exists $\de>0$ such that: 
$$z \in \Om, x \in K, |z-x|<\de \Rightarrow  |u(z)-\phi (x)|<\ve.$$ 
Hence $\phi$ can be approximated uniformly on $K$ by continuous plurisubharmonic functions of the form $(u* \rho_\de)(z+ t\bold n_{z_0})$ where $t, \de$ are appropriate parameters converging to $0,$
where $u*\rho_\de$ denotes the convolution of $u$ with the standard smoothing kernels 
$\rho_\de (z):=\fr1{\de^{2n}} \rho(z)$ and $\rho \ge 0$ is a smooth, radial function with support in the unit ball and with
$\int_{\C^n} \rho dV_{2n}=1.$
Thus the compact set $K$ is $B-$regular for $r_0$ small enough. Hence $\pa \Om$ is locally $B-$regular at $z_0.$
This shows that $(\partial \Omega)\setminus X$ can be expressed as an increasing union of $B$-regular compact sets. Consequently, by the assumption on $X$, Proposition 1.9 in \cite{Si} applies, yielding that $\partial \Omega$ is $B$-regular, since it is a countable union of such compact sets.
It suffices now to apply Theorem 2.1 in \cite{Si} to conclude that $\Om$ is $B-$regular.
\end{proof}
\section{Envelopes on quasi B-regular domains}
In this section we prove the main results announced in the introduction and, at the end, present examples to demonstrate their applicability.
	\begin{proof} (of Theorem \ref{breg})
		(a) Let $\{\phi_j\}_{j \ge 1}$ be a countable dense subset of $\c (\pa \Om).$ Then there exist quasi upper bounded functions $u_j \in PSH(\Om)$
		and $\Om-$pluripolar sets $A_j \subset \pa \Om$ such that
		\begin{equation} \label{eq010}
			\lim\limits_{z \to x} u_j(z)=\phi_j (x), \ \forall x \in (\pa \Om) \setminus A_j.
		\end{equation} 
		Set $$A:= \bigcup_{j \ge 1} A_j.$$ Then, by Lemma \ref{union}, $A \subset \pa \Om$ is a $\Om-$pluripolar set.
		Since $\Om$ is regular, we see that 
		$$\tilde P_{\phi_j} = \sup \{u: u \in PSH(\Om), u^*|_{\pa \Om} \le \phi_j \}  \in PSH(\Om) \cap L^\infty (\Om)$$ 
		and  
		\begin{equation} \label{eq0}
			\tilde P_{\phi_j} \le \phi_j \ \text{on}\  \pa \Om.
		\end{equation}
		We claim that
		\begin{equation} \label{eq1}
			\lim_{z \to x} \tilde P_{\phi_j} (z)=\phi_j (x), \ \forall x \in (\pa \Om) \setminus A_j.
		\end{equation}
		For this, fix $\ve>0$ and $w_j \in PSH(\Om), w_j<0$
		such that $w_j^*|_{A_j}=-\infty.$ Now,  since $u_j$ is quasi upper bounded, thee exists a sequence $u_{j,l} \subset PSH(\Om) \cap L^\infty (\Om)$
		such that $u_{j,l} \nearrow u_j$ quasi everywhere on $\Om$
		as $l \to \infty$. Since $u_{j,l}^* \le u_j^*$ on $\pa \Om$ for every $l,$
		we infer, in view of (\ref{eq010}) and the choice of $w_j$, that
		$u_{j, l,\ve}^* \le \phi_j$ on $\pa \Om,$ where 
		$$u_{j, l, \ve}:= u_{j,l}+\ve w_j \in PSH(\Om).$$ 
		Hence
		$$u_{j, l,\ve} \le \tilde P_{\phi_j}, \ \text{on}\ \Om.$$
		By letting $\ve \to 0,$ 
		we obtain $u_{j,l} \le \tilde P_{\phi_j}$ on $\Om \setminus w_j^{-1} (-\infty).$
		Since $u_{j,l} $ and $\tilde P_{\phi_j}$ are \psh\ on $\Om,$ we see that $u_{j,l} \le \tilde P_{\phi_j}$ on $\Om.$ Thus, by passing $l \to \infty$ we get  $u_j \le \tilde P_{\phi_j}$ quasi everywhere, and hence entirely on $\Om.$
		This yields that
		$$\liminf_{z \to x} \tilde P_{\phi_j} (z) \ge \lim_{z \to x} u_j (z)=\phi_j (x), \ \forall x \in (\pa \Om) \setminus A_j.$$
		Coupling this with (\ref{eq0}) we finish the proof of our claim (\ref{eq1}).
		
		Now, let $\phi \in \c(\pa \Om)$ be an arbitrary function. We can extract a subsequence \( \phi_{j_k} \) that converges uniformly to \( \phi \) on \( \partial \Omega \). Consequently, \( \tilde P\phi_{j_k} \) converges uniformly to \( \tilde P\phi \) on \( \Omega \). Now, fix $x_0 \in (\pa \Om) \setminus A$. Then, by (\ref{eq1}) we have 
		$$\lim_{z \to x_0} \tilde P_{\phi_{j_k}} (z)= \phi_{j_k} (x_0).$$
		Thus, by letting $k \to \infty$ and using the uniform convergence of \( \tilde P\phi_{j_k} \) we infer that
		$$\lim\limits_{z \to x_0} \tilde P_{\phi} (z)= \phi (x_0).$$
		It remains to treat the second assertion in (\ref{eq000}). For this purpose, since the space of $\mathcal C^1$ smooth functions
		on $\C^n$ is dense in $\c (\Om),$ using an approximation argument similar to the above,
		we may assume $\phi$ is $\c^1$ smooth on $\C^n.$
		Next,
		we let 
		$$h (z):=-|z-x_0|,\ z \in \C^n.$$
		From the results established above, we conclude that:
		
		\n 
		(i) $\tilde P_{h}:= \sup \{u: u^*|_{\pa \Om} \le h\} \in PSH(\Om) \cap L^\infty (\Om);$
		
		\n 
		(ii) ${\tilde P_h}^* \le h$ on $\pa \Om;$
		
		\n 
		(iii) $\lim\limits_{z \to x_0} \tilde P_{h} (z)=h(x_0)=0.$
		
		\n 
		By applying the maximum principle to $\tilde P_h-h$ we see that $\tilde P_h<h$ on $\Om.$ Since
		$\va$ is $\c^1$ smooth, we may choose $\la>0$ so large such that 
		$$-\la h(z)=\la |z-x_0| \ge |\phi (z)-\phi (x_0)| \ \forall z \in \ov{\Om}.$$
		This yields that
		$$\la \tilde P_h(z)+\va(x_0)<\la h(z)+\va(x_0) \le \phi (z) \ \forall z \in \Om.$$
		Hence $\tilde P_\phi \ge \la \tilde P_h+\phi (x_0)$ on $\Om.$ This implies, in view of (iii), that 
		$$\liminf_{z \to x_0} P_\phi (z) \ge \la \liminf_{z \to x_0} \tilde P_h (z)+\phi(x_0)=\phi(x_0).$$
		Since $P_\phi \le \phi$ on $\Om,$ we conclude that $\lim\limits_{z \to x_0} P_\phi (x_0)=\phi (x_0)$ as desired.
		
		\n 
		(b) The assertion (i) is obvious, (ii) is the second assertion in (\ref{eq000}). 
		Set $X=\hat A$ where $\hat A$ is the $\Om$-pluripolar hull of $A$
		To finish the proof, we will show that $P_\phi$ is continuous at every point of $\Om \setminus X.$ 
		Suppose on the contrary that $P_\phi$ is discontinuous at some point $z^* \in \Om \setminus \hat A.$ Then there exists a sequence
		$\Om \ni \{z_j\}_{j \ge 1} \to z^*$ and $\de>0$ such that 
		\begin{equation} \label{eq5}
			P_\phi (z_j)<P_\phi (z^*)-2\de \ \forall j \ge 1.
		\end{equation} 
		Let $v \in PSH(\Om), v<0$ be such that $v^* \equiv -\infty$ on $A$ but $v(z^*)>-\infty.$ Fix $\la>0.$
		For each $\ve>0$ we	let $$\Om_\ve:= \{z \in \Om: \text{dist} (z, \pa \Om)>\ve \}.$$
		Since $\phi \in \c(\ov \Om)$ we may
		choose $\ve_0$ such that 
		$$|\phi (z)-\phi(z')|<\de \ \forall z, z' \in \ov \Om \ \text{with}\ |z-z'|<\ve_0.$$
		Our key claim is the following assertion:
		There exists $\ve \in (0, \ve_0)$ so small such that 
		$$P_\phi (z+\eta)+\la v(z+\eta)-\de \le P_\phi (z) \ \forall z \in \pa \Om_\ve, \forall \eta \in \mathbb B(0,\ve).$$
		If the claim is false, then there exist a sequence $\ve_j \searrow 0$, points $\xi_j \in \pa \Om_{\ve_j}$ 
		and $\eta_j \to 0 \in \C^n$ 
		such that
		\begin{equation} \label{eq7} 
			P_{\phi}  (\xi_j+\eta_j)-\de>  P_{\phi} (\xi_j+\eta_j)+\la v(\xi_j+\eta_j)-\de>P_\phi (\xi_j).
		\end{equation}
		After switching to a subsequence if necessary, we may assume that $\xi_j \to \xi \in \pa \Om.$
		Since $v^*=-\infty$ on $A$ we infer that $\xi \not \in A.$ But, then by combining (ii) and (\ref{eq7}) we also reach a contradiction.
		Thus the claim is proved. 
		Hence,  for $\eta \in \mathbb B(0,\ve)$  the gluing function
		$$\Phi (z):= \begin{cases}
			\max \{P_\phi (z+\eta)+\la v(z+\eta)-\de, P_\phi (z) \} & z \in \Om_\ve\\ 
			P_\phi (z) & z \in \Om \setminus \Om_\ve
		\end{cases}$$
		belongs to $PSH(\Om)$ and, since $\ve<\ve_0$ we can check that $\Phi \le \phi$ on $\Om.$
		Thus $\Phi \le P_\phi$ on $\Om$. It follows that 
		$$P_\phi (z+\eta)+\la v(z+\eta)-\de \le P_\phi (z) \ \forall z \in \Om_\ve.$$
		In the above estimate, by letting $z:= z_j, \eta:= z^*-z_j$ for $j$ large enough we obtain 
		$$P_\phi (z^*)+\la v(z^*)-\de \le P_\phi (z_j).$$
		Combining the above estimates and (\ref{eq5}) we get $\la v(z^*) \le -\de.$ Since $\la>0$ can be chosen arbitrarily small and since $v(z^*)>-\infty$ we obtain a contradiction. We are done.	
	\end{proof}
	\begin{proof} (of Theorem \ref{main2})
		(i) Our key step is to show $P_{{ \phi^*}}$ is locally upper bounded on $\Om$. To this end, we note that, since $v$ is bounded from below on $\ov{\Om}$, after adding a large constant to $v$ we may achieve that $v>0$ on $\ov{\Om},$ and  $\phi^* \le v$ on $\Om$, in view of Lemma 2.7. Moreover, by replacing $\phi$ with $\max \{\phi, 0\}$ we can even suppose
		$\phi \ge 0$ on $\Om.$
		Now, fix $u\in PSH(\Om)$ in the defining family of $P_{ \phi^*}$. Then, on $\Om$
		we have $u\le \phi^*\le v$.
		For each $j\ge1$, we define $$\phi_j=\min\{\phi^*, j\}.$$
		Then $\phi_j \ge 0$ is upper bounded on $\overline{\Omega}$, continuous on $\overline{\Omega} \setminus E_\phi$.
		Since $\Omega$ is regular in the real sense, it is known that (see for instance Theorem 1.2.7 in \cite{Blocki}),
		the classical Perron envelope
		$$
		h_j:=h_{\phi_j}=\sup\{u\in SH(\Omega), u^*|_{\partial \Omega}\le \phi_j|_{\partial \Omega}\}
		$$
		is non-negative, harmonic in $\Omega$ with the boundary value $\phi_j$ on  $\partial \Omega\setminus E_\phi$.
		It follows that
		$$
		h_j = \phi_j\le \phi^*=\phi_* \le v_* \ \text{on}\ \pa \Om \setminus E_\phi.
		$$
		So $h_j\le v$ on $\Omega$ by Lemma \ref{extended maximal} (ii). 
		Since $h_j\le h_{j+1}$ on $\Omega$ for all  $j\ge1$, and since $h_j$ is dominated by $v \not \equiv+\infty$
		using Harnack's theorem, $h_j$ increases to a harmonic function  $0 \le h \le v$ on $\Omega$.
		Now for each $z_0\in\partial \Omega \setminus E_\phi$ and  $j\ge1$, we have
		$$
		\phi_j(z_0)=\lim_{z\to z_0}h_j(z)\le \liminf_{z\to z_0} h(z) =h_*(z_0).
		$$
		Letting $j\to \infty$ we obtain
		\begin{equation}\label{1}
			u^*(z_0) \le  \phi^* (z_0)\le h_*(z_0),\ \forall z_0\in \partial \Omega\setminus E_\phi.
		\end{equation}
		Using again Lemma \ref{extended maximal} (ii) we obtain $u \le h$ on $\Omega$.
		Since $u$ is arbitrary, $P_{ \phi^*}\le h$ on $\Omega$.
		So $P_{\phi^*}$ is locally bounded from above and hence, this function coincides with its upper regularization
		on $\Om.$ Thus $P_{\phi^*} \in PSH(\Om)$.
		
		\n 
		(ii) Since $\theta \in PSH^* (\Om)$ we infer that $P_{\phi^*} \ge \theta$. It follows that $P_{\phi^*} \in PSH^* (\Om)$
		as well. Moreover, by Lemma \ref{majorant lemma} we conclude that $P_{\phi^*}$ is quasi upper bounded and admits $v$ as a strong plurisuperharmonic majorant.

		\n 
		(iii) By adding a constant we may assume $\theta>0$ on $\ov \Om.$ This yields $v_* \ge \va_*\ge 0$ on $\ov \Om.$
		Next, since $\phi_*$ is a non negative lower semicontinuous on $\overline{\Omega}$,
		we can find an increasing sequence of  non negative continuous functions $\varphi_j \nearrow \phi_*$ on $\overline{\Omega}$.
		Since $\Omega$ is quasi B-regular, by Theorem \ref{breg} (b), we have
		$P_{\varphi_j} \in PSH(\Om)$, continuous on $\Omega \setminus \hat B_\Om$ and satisfies  
		$$\lim\limits_{z\to z_0}P_{\varphi_j}(z) =\varphi_j(z_0), \ \forall z_0 \in (\pa \Om) \setminus B_\Om.$$
		Then for $z_0 \in \pa \Om \setminus (B_\Om \cup E_\phi)$
		and $j \ge 1$, we have
		$$\liminf_{z \to z_0} P_{\phi^*} (z) \ge \liminf_{z \to z_0} P_{\phi_*} (z) \ge \lim_{z \to z_0} P_{\va_j} (z_0) =\va_j (z_0).$$
		By letting $j\to \infty$ we conclude that  
		$$\lim\limits_{z\to z_0}P_{\phi^*}(z) =\phi(z_0) \ \forall z_0 \in \pa \Om \setminus (B_\Om \cup E_\phi).$$
		We should say that, some idea given in Theorem 4.2 of \cite{NiWi} has been used in the forgoing proof.
		Now we have to work a bit harder to obtain continuity of $P_{\phi^*}$ in the interior.
		To this end, take  $w\in PSH^-(\Omega)$ such that  $w^*|_{E_\phi}=-\infty$.
		For $\delta>0$  we define 
		$$
		u_{\delta}=P_{\phi^*} +\delta(w -v) \in PSH(\Om).
		$$
		Then
		\begin{equation}\label{2}
			u^*_\delta\le (P_{\phi^*}-\delta v)^*+\de w^*  \ \text{on}\ \overline{\Omega}.
		\end{equation}
		It then follows from  Lemma \ref{majorant lemma}  that $u^*_\delta$ is upper bounded, upper semicontinuous on $\overline{\Omega}$.
		We claim that $u^*_{\delta}\le\phi_*$ on $\overline{\Omega}$.
		Fix $z_0\in \overline{\Omega}$.  If $z_0\notin E_\phi $ then
		we  have
		$$\begin{aligned}
			u^*_{\delta}(z_0)
			&\le P^*_{\phi^*}(z_0)+\delta w^*(z_0)-\delta v_*(z_0)\\
			&\le   P^*_{\phi^*} \ (\text{since}\  w^* \le 0, v_* \ge 0)\\
			&\le \phi^*(z_0)=\phi_*(z_0) \ (\text{since}\  z_0 \not \in E_\phi).
		\end{aligned}
		$$
		If $z_0\in E_\phi$ then $w^*(z_0)=-\infty,$ using (\ref{2}) we obtain  $u^*_\delta(z_0)=-\infty <\phi_*(z_0).$
		The claim follows.\\
		Now since  $u^*_{\delta}$ is upper semicontinuous  and  $\phi_*$ is lower semicontinuous on $\overline{\Omega}$,
		it follows from the insertion lemma (cf. Lemma \ref{insertion}) that there exists a real valued continuous function $\phi_{\delta}$ on $\overline{\Omega}$
		such that
		$$
		u^*_{\delta}\le \phi_{\delta}\le \phi_*
		$$
		on $\overline{\Omega}$. Then, for $\de>0,$ it follows from  the definition of  envelopes  that
		\begin{equation} \label{eq111} 
			P_{\phi^*} +\delta(w -v)=u_{\delta}\le P_{\phi_{\delta}}\le P_{\phi_*} \le P_{\phi^*} \ \text{on}\ \ \Omega.
		\end{equation}
		Set
		$$A:=\{z \in \Om: v (z)=+\infty \} \cup (\widehat {E_\phi} \cap \Om).$$
		Then $A$ is pluripolar and 
		$$P_{\phi^*}=\limsup_{\de \to 0} P_{\phi_\de} \ \text{on}\ \Om \setminus A.$$
		Observe that, by Theorem \ref{breg} (b), each $P_{\phi_\de}$ is continuous at every point in $\Om \setminus X$.
		Now suppose that $P_{\phi^*} $ is discontinuous at some point
		$z^* \in \Omega \setminus Y$ where $Y:=X \cup A$. Then there exist $\eta>0$ and a sequence $\Om \ni z_j \to z^*$
		such that
		\begin{equation} \label{eq11}
			P_{\phi^*} (z^*)>P_{\phi^*} (z_j)+\eta \ \forall j.
		\end{equation}
		Since $Y$ has empty interior and since $P_{\phi^*}$ is upper semicontinuous on $\Om,$ we may assume $z_j \in \Omega \setminus Y$ for every $j.$
		Choose $\de_0 >0$ such that 
		\begin{equation} \label{eq12}
			P_{\phi^*} (z^*)<P_{\phi_{\de_0}} (z^*)+\fr{\eta}2.
		\end{equation}
		Combining (\ref{eq11}) and (\ref{eq12}) we obtain
		$$P_{\phi_{\de_0}} (z^*)>P_{\phi^*} (z_j)+\fr{\eta}2 \ \forall j.$$
		On the other hand, using (\ref{eq111})  we get $$P_{\phi^*} (z_j) \ge P_{\phi_{\de_0}} (z_j).$$
		Putting these last two estimates we arrive at 
		$$P_{\phi_{\de_0}} (z^*)> P_{\phi_{\de_0}} (z_j)+\fr{\eta}2, \ \forall j.$$
		By letting $j \to \infty$, we obtain a contradiction to the fact that $P_{\phi_{\de_0}}$ is continuous at $z^*$.
		Thus $P_{\phi^*}$ is continuous at every point in $\Om \setminus Y$ as desired.
	\end{proof}	
	\n
	Next, for the proof of Theorem \ref{main3} we require the following fact.
	\begin{lemma} \label{smajorant}
		Let $\psi: \ov{\Om} \to (-\infty, \infty]$ be a nearly continuous, lower semicontinuous function such that $\psi$ is plurisuperharmonic on $\Om.$ Let $w \in PSH^* (\Om), w<0$ be such that $w^*|_{E_\psi} \equiv -\infty.$
		Then $\psi-w^*$ is a strong plurisuperharmonic majorant of 
		$(\psi^{1/\al})^{*}$ on $\ov \Om$ for every $\al>1.$
	\end{lemma}
	\begin{proof}
	Fix $\al>1.$ Then, for every $\de>0,$ by the proof of Lemma \ref{majorant lemma} (i), we can find a constant $N_\de>0$ such that
	$$(\psi^{1/\al})^{*}=(\psi^*)^{1/\al} \le \de \psi^*+N_\de \ \text{on}\ \ov{\Om}.$$
Since $\psi^* \le \psi-w$ on $\ov{\Om}$, we conclude that $(\psi^{1/\al})^{*} \le \de (\psi-w) +N_\de$ on $\ov \Om.$ This proves the lemma.
		\end{proof}
	\begin{proof} (of Theorem \ref{main3}) Our approach adapts ideas from Theorem 3.1 in \cite{DW}.
	First, we observe that, since $\alpha>1$ and $\psi \ge 0$ is plurisuperharmonic on $\Om,$ the function $\psi^{1/\alpha}$
	is also plurisuperharmonic on $\Om.$ Hence, by the hypothesis, $\psi^{1/\alpha}$ is a nearly continuous plurisuperharmonic majorant of $u^*$ on $\ov{\Om}.$  We also fix $w \in PSH^* (\Om), w<0$ satisfying $w|_{E_\psi} \equiv -\infty.$
	
\n 
Now, assume $u$ is bounded from below. Then, by the preceding arguments, 
	$u^*-\psi^{1/\alpha}$ is upper semicontinuous on $\ov \Om$ and bounded from above. Thus, there exists a sequence $f_j$ of real valued continuous functions on $\ov \Om$ that decreases to $u^*-\psi^{1/\alpha}.$ 
	By Dini's theorem we can assume that
	$$\max_{\ov \Om} f_j <1+\max_{\ov \Om} (u^*-\psi^{1/\alpha}), \ \forall j.$$
	Set 
		$$\phi_j:= f_j+  \psi^{1/\alpha}.$$	
		Then $\phi_j \searrow u^*$ on $\ov \Om$
	and	$E_{\phi_j}=E_\psi$ is $\Om-$pluripolar, so $\phi_j$ is nearly continuous. 
	Next, by Lemma \ref{smajorant}, $\psi-w^*$ is a {\it strong} plurisuperharmonic majorant of
	$\phi_j^*.$ 
	
	 Consider the envelopes
		\begin{equation*} \label{eq3}
			u_j (z):=P_{\phi_j^*}(z) =\sup\{v(z): v\in PSH(\Omega): v^*\le \phi_j^* \ \text{on}\ \overline{\Omega}\}.
		\end{equation*}
Since $\phi_j \ge u$  is
		bounded from below on $\Om,$ using Theorem \ref{main2} (iii), we obtain:
		
		\n 
		(a) $u_j \in PSH^* (\Om);$ 
		
		\n
		(b) $u_j$ are quasi upper bounded and decreasing on $\Om;$
		
		\n 
		(c) $u \le u_j \le  \phi_j \le 1+\max\limits_{\ov \Om} (u^*-\psi^{1/\alpha}) +\psi^{1/\al}$ on $\Om$ for every $j.$
		
		\n 
		(d) $u_j$ are real valued continuous on $\Om \setminus Y$, where $$Y:=B_\Om \cup \{z \in \Om: \psi=+\infty\} \cup E_\psi.$$   
\n 
It follows that $u_j$ decreases to $u$ on $\Om$ and has the desired continuity on $\Om \setminus Y$.

\n 		
For general $u$, we will use the method in Theorem 4 of \cite{FW}. By the previous case, for each $m \ge 1$ we can find 
a sequence $u_{k,m}  \subset PSH^*(\Om)$, continuous on $\Om \setminus Y,$ decreases to $\max \{u, -m\}$ on $\Om$ as $k \to \infty$ and satisfies the estimates
\begin{equation} \label{12}
u_{k, m} \le  1+\max\limits_{\ov \Om} (\max\{u, -m\})^*-\psi^{1/\alpha}) +\psi^{1/\al}
\end{equation}
Now we write the open set $\Om \setminus Z$ as an increasing union of compact sets $K_m$ such that $K_m$ is included in the interior of $K_{m+1}.$
It follows that the sequences $\max \{u_{k,m}, u_{m, l}\} \searrow u_{m, l}$ on each compact set $K_l$ as $k \to \infty.$
So, by Dini's theorem, for each $m \ge 1,$ we can find $k(m)$ so large such that
\begin{equation} \label{eq123}
u_{k(m), m} \le u_{m, l}+\fr1{m} \ \text{on}\  K_l, l \in \{1, 2, \cdots, m\}.
\end{equation}
Let
$$u_j: =\sup_{m \ge j} u_{k(m),m} \ \text{on}\ \Om.$$
Then obviously $u_j \searrow u$ on $\Om \setminus Z.$
Moreover, by (\ref{12}), we infer that, for each $j,$ the function $u_j$ is dominated by a constant plus $\psi^{1/\al}$ on $\Om$. 
So, apply again Lemma \ref{smajorant}, we deduce that $\psi-w$ is a strong plurisuperharmonic majorant for $u_j^* \in PSH^* (\Om).$
Hence, $u_j^*$ is quasi upper bounded for very $j$, in view of Lemma \ref{majorant lemma} (ii).

To establish continuity of $u_j^*$ on $\Om \setminus Z,$
it suffices to show $u_j$ is upper semicontinuous on $\Om \setminus Z.$ Fix $j \ge 1, z_0 \in \Omega \setminus Z$ and $\ve>0.$ Choose $p>j$ so large that $z_0 \in K_p$ and $1/p<\ve.$ Then for $m \ge k(p)$ and points $z \in K_p,$ 
by (\ref{eq123}) we have 
$$u_{k(m), m} (z) \le u_{m, p} (z)+1/m \le  u_{k(p), p} (z)+ \ve.$$
This implies that 
$$\limsup_{z \to z_0} \Big( \sup_{m \ge k(p)}  u_{k(m),m} (z) \Big) \le u_{k(p), p} (z)+ \ve.$$
Hence
$$\limsup_{z \to z_0} \Big (\sup_{m \ge j}  u_{k(m),m} (z) \Big) \le u_{j} (z)+ \ve.$$
Thus $u_j$ is upper semicontinuous on $\Om \setminus Z.$ On the other hand, being supremum a family of continuous function,
$u_j$ is lower semicontinuous on $\Om \setminus Z.$ Thus $u_j$ is continuous on $\Om \setminus Z.$
Observe that  $u_j^* \in PSH^* (\Om)$ and by the above reasoning we have $u_j^*=u_j$ on $\Om \setminus Z.$ Hence $u_j^*$ is continuous on $\Om \setminus Z.$
Notice that $u_j^* \searrow \tilde u \in PSH(\Om).$ It follows that $\tilde u=u$ on $\Om \setminus Z,$ 
and since $Z$ is pluripolar we infer that 
$\tilde u=u$ on $\Om.$ So $u_j^* \searrow u$ on $\Om.$ The proof is complete.	\end{proof}
\begin{remark} {\rm (a) We do not know if Theorem \ref{main4} still holds in the case $\al=1.$
	
\n 
(b) Under the extra hypothesis that \(u\) is bounded from below but without assuming the existence of $Z$, our proof yields an
	approximating sequence \(\{u_j\}\) with the additional property that each \(u_j\) admits boundary limits on
	\(\partial\Omega\setminus B_\Omega\), i.e.,
	\[
	\lim_{z\to x} u_j(z)=u_j^*(x)\quad(x\in\partial\Omega\setminus B_\Omega),
	\]
	and moreover \(u_j^* \searrow u^*\) on \(\partial\Omega\setminus B_\Omega\).
}	
\end{remark}
	\begin{proof} (of Theorem \ref{main4})
		First, we show $P_{\phi^*} \in MPSH(\Om)$. For this,
		let $G$ be a relatively compact open subset of $\Omega$ and $u\in PSH(G)$ such that $u^*\le P_{\phi^*}$ on $\partial G$. Then the function
		$$ \tilde{u}=
		\begin{cases}
			\max\{u, P_{\phi^*}\}, & \mbox{on } G \\
			P_{\phi^*}, & \mbox{on}\  \Omega\setminus G.
		\end{cases}
		$$
		is plurisubharmonic on $\Omega$. On $\partial G\setminus E_\phi$ we have
		$$
		u\le P_{\phi^*}\le\phi^*=\phi_*.
		$$
		Thus  the subharmonic function $u-\phi\le 0$ on $\partial G$ outside the $\Omega$-polar set $E_\phi$.
		By Lemma \ref{extended maximal} (ii) we have $u\le\phi\le\phi^*$ on $G$.
		Hence $\tilde{u}\le \phi^*$ and therefore $\tilde{u}$ belongs to the defining class of $P_{\phi^*}$, so  $\tilde{u}\le P_{\phi^*}$.
		In particular, $u\le P_{\phi^*}$ on $G$. So $P_{\phi^*} \in MPSH(\Om)$ as desired. Next, by Theorem \ref{main2}(iii) we see that $P_{\phi^*}$ is bounded from below and has the correct boundary values $\phi$ on 
$\partial \Omega\setminus (B_\Om \cup E_\phi).$ 

\n		
Finally, the uniqueness of the solution follows immediately from Lemma \ref{extended maximal} (i).
\end{proof}	
	\n 
	We now recover the following result of Nilsson and Wikström, originally stated in \cite{NiWi} and mentioned in the introduction.
	\begin{corollary}
		Assume that $\Omega\subset\mathbb{C}^n$ is a bounded B-regular domain, and  let $\phi$ be a lower bounded, tame plurisuperharmonic function on $\Omega$ such that $\phi^*=\phi_*$ on $\overline{\Omega}$.
		Then the associated Perron-Bremermann envelope
		$$
		P_{\phi}=\sup\{u(z): u\in PSH(\Omega): u^*\le\phi_*\}
		$$
		is a maximal plurisubharmonic function that is  continuous outside a pluripolar set.
		Furthermore, $P_\phi$ is the unique  maximal plurisubharmonic function with the correct boundary value $\phi$, i.e.  for $z_0\in\Omega$,
		$$
		\lim\limits_{\Omega\ni \xi \to z_0}P_\phi(\xi)=\phi(z_0).
		$$
	\end{corollary}
	\begin{remark} {\rm (i) Let $\Omega=\mathbb{D}$ be the unit dics in $\mathbb{C}$. Consider the Poisson kernel
			$$h(z):= \frac{1-|z|^2}{|1-z|^2}, \ \ z\in \mathbb{D}.$$
			Then $h>0$ and harmonic in $\Om$ and obviously $h^* \equiv 0$ on $\pa \mathbb D \setminus \{1\}.$
			It implies that $h_* (1)=0.$
			On the other hand, by direct computations we obtain
			$$\lim_{n \to \infty} h \Big (1-\fr1{n} \Big)=+\infty \Rightarrow h^* (1)=+\infty.$$		
			Thus $E_h=\{1\}.$ Now let $\chi_1, \chi_2$ be $\c^2-$smooth functions on $(0, \infty)$ and satisfies: 
			
			\n 
			(a) $\lim\limits_{x \to 0} \chi_1 (x)=\lim\limits_{x \to 0} \chi_2 (x)=0;$
			
			\n 
			(b) $\lim\limits_{x \to +\infty} \chi_1 (x)=\lim\limits_{x \to \infty} \chi_2 (x)=+\infty;$
			
			\n 
			(c) $\chi_1'' \le 0, \chi_2''\le 0$ on $(0, \infty);$
			
			\n 
			(d) $\lim\limits_{x \to \infty} \fr{\chi_2 (x)}{\chi_1 (x)}=+\infty.$
			
			\n
			For instance, we may take $\chi_1 (x)=x^{\alpha}, \chi_2 (x)=x^{\beta}, 0<\alpha<\beta<1.$

			\n
			Then $\phi:=\chi_1 (h)>0$ and is superharmonic on $D$ satisfying 
			$$\phi^* (1)=+\infty, \phi_* (1)=0 \Rightarrow E_\phi=\{1\}.$$
			Moreover, $\phi$ admits $\psi:=\chi_2 (h)$ as a strong superharmonic  majorant. 
			Thus, by Theorem \ref{main2} (ii), there exists a unique quasi upper bounded harmonic function $u$ on $\mathbb D$ with the right boundary values $\phi$ along $\pa \mathbb D \setminus \{1\}.$
			
			\n 
			(ii) We now construct an explicit solution of the Dirichlet problem in high dimension.
			For $\al \in (0,1)$ consider the ellipsoid
			$$\Om_\al:= \{(z, w): |z|^2+ |w|^{1/\alpha} <1\}.$$
			It is easy to check that $\Om_\al$ is a hyperconvex Reinhardt domain whose boundary contains no complex variety of positive dimension. So, by Proposition 2.1 in \cite{DDH}, $\Om_\al$ is $B-$regular. Now we set 
			$$\begin{aligned} 
				\phi_\al (z,w)&:= h(z)^\alpha\\ 
				u_\al (z,w)&:=\frac{|w|}{|1-z|^\al},\ \forall (z, w)\in \mathbb{B}.
			\end{aligned}
			$$
			\n 
			Then, since $\Om_\al \subset \mathbb D \times \C$, by (i), $\phi$ is superharmonic on $\Om_\al$ admitting a strong majorant on $\ov{\Om_\al}.$
			On the other hand,
			{\it locally} on $\Om_\al, u_\al$ is the modulus of a holomorphic function, hence $u_\al \in PSH(\Om_\al).$
			Moreover, since the restriction of $u_\al$ on each complex disk $w=c(1-z)^\alpha, c \in \C$ is constant, we infer that
			$u_\al$ is maximal on $\Om_\al.$ So, by Theorem \ref{main2} (ii),
			$u_\al$ is the unique quasi upper bounded maximal plurisubharmonic function on $\Om_\al$ having the boundary values  $\va$ on $\pa \Om_\al \setminus \{(1,0)\}.$
			
			\n 
			(iii) If $\phi$ does not assume to have a strong majorant then the uniqueness of the solution may fail.
			Indeed, following \cite{NiWi}, we consider on the unit ball $\Omega=\mathbb{B} \subset \mathbb{C}^2$ the function
			$$
			\phi(z,w)=\frac{1-|z|^2}{|1-z|^2},\  (z,w)\in \mathbb{B}.
			$$
			Then obviously $P_{\phi^*}=\phi$ on $\mathbb{B}.$
			Notice that $E_\phi=\{(1,0)\}.$
			On the other hand
			$$
			V(z,w)=\frac{|w|^2}{|1-z|^2},\  (z,w)\in \mathbb{B}
			$$
			is also maximal plurisubharmonic in $\mathbb{B}$ and  $V=P_{\phi^*}$ on $\partial \mathbb{B}\setminus \widehat{E_\phi}$. Nevertheless,  $V < P_{\phi^*}$ everywhere on $\mathbb{B}$.
			
			\n 
			(iv) 
			Consider the function
			$\va= 0$ on $\pa \mathbb D \setminus \{1\}$ and $\va (1)=+\infty.$ Then there exists no harmonic function on $\mathbb D$ having boundary values equal to $\phi$ everywhere on $\pa \mathbb D.$
			For otherwise, by the minimum principle, such a solution $u$ must be positive on $\mathbb D.$ By Poisson-Herglotz integral formula we may write
			$$u(z)=\int_{\pa \mathbb D} \fr{1-|z|^2}{|z-\xi|^2} d\mu(\xi),$$
			where $\mu$ is a positive finite measure on $\pa \mathbb D.$ Since $u$ tends to $0$ at every boundary point different from $1$, we see that $\mu$ has no atom on $\pa \mathbb D \setminus \{1\}.$ Furthermore, using Radon-Nikodym theorem, we can check that $\mu$ equals to its singular part $\mu_s$  with respect to $d\sigma,$ the Lebesgue measure on $\partial \mathbb D.$ Nevertheless, using Lusin-Privalov theorem about boundary values of Poisson integral with respect to singular measures, we see that $\mu_s$ may concentrate only at $1$. Thus $\mu$ is a multiple of the Dirac mass $\delta_1.$
			So $$u(z)=c\fr{1-|z|^2}{|z-1|^2}$$ for a constant $c>0.$ This is absurd since, in this case, $u(z) \to 0 $ as $z \to 1$ tangentially. 		
		}
		
	\end{remark}



\end{document}